\documentclass[12pt,a4paper]{article}
\usepackage[utf8]{inputenc}
\usepackage[english]{babel}
\usepackage{amsmath}
\usepackage{amsfonts}
\usepackage{graphicx}
\usepackage{fouridx}
\usepackage{amsthm}
\usepackage{amssymb, amscd}
\usepackage{enumerate}
\usepackage{mathtools}
\usepackage[french]{varioref}
\usepackage{multirow} 
\usepackage{hyperref}
\usepackage{bm}
\usepackage{multicol}
\usepackage{multirow}
\author{ Saadaoui Nejib \thanks{ Université de Gabès, Laboratoire
		Mathématiques et Applications. Faculté des Sciences de Gabès
		Cité Erriadh 6072 Zrig Gabès Tunisie.
,\qquad \textbf{nejib.saadaoui@fsg.rnu.tn}}}
\title{ (Bi)Hom–Leibniz algebra}
\newtheorem{theorem}{Theorem}[section]
\newtheorem{cor}[theorem]{Corollary}
\newtheorem{thm}[theorem]{Theorem}
\newtheorem{lem}[theorem]{Lemma}
\newtheorem{prop}[theorem]{Proposition}
\newtheorem{example}{Example}[section]
\newtheorem{defn}[theorem]{Definition}
\newtheorem{remq}[theorem]{Remark}
\numberwithin{equation}{section}
\usepackage{multicol}
\input{epsf.sty}

\usepackage{varioref}

\newcommand{\Z}{\mathbb{Z}}

\newcommand{\K}{\mathbb{K}}

\begin{document} 
	\maketitle

	\begin{abstract}
 The first aim of this paper is to introduce and study  symmetric (Bi)Hom-Leibniz algebras, which are
left and right Leibniz algebras.  We discuss $\alpha^k\beta^l$-generalized derivations, $\alpha^k\beta^l$
-quasi-derivations and 
$\alpha^k\beta^l$-quasi-centroid of (Bi)Hom-Leibniz algebras and colour BiHom-Leibniz algebras. 
The second aim  is to define  a new type of BiHom-Lie algebras satisfies the following hierarchy 
\begin{equation*}
\displaystyle \{ \text{   BiHom-Lie type $B_1$}  \}\supseteq_{\beta=id}   \{ \text{   Hom-Lie}  \}\supseteq_{\alpha=id} \{ \text{   Lie}  \}. 
\end{equation*}  
 Moreover, define representations and a cohomology of symmetric  BiHom- Leibniz algebras of type $B_1$.
	\end{abstract}
\section*{Introduction}	
In 1993, J.L. Loday introduced  Leibniz algebras which are a generalization of Lie
algebras \cite{Loday1,Loday2}. They are defined by a bilinear bracket which is no longer skew-symmetric. More precisely,	
A left Leibniz  algebra over a a field $\K$  
is a $\K$-vector space $A$  with a $\K$-bilinear map $[\cdot,\cdot] \colon A\times A\to  A $ satisfying 
\begin{equation}\label{equtionL}
\displaystyle  \big[a,[ b,c ]  \big]=\big[[a , b],c   \big] +  
\big[b ,[a,c]   \big],
\end{equation}
for all $a,\, b,\, c\in A$.\\
This property means that, for each $a$ in $A$, the adjoint endomorphism of $A$, $ad_a=[a,\cdot]$ is a derivation of 
$(A,[\cdot,\cdot])$.

Similarly,  a right Leibniz algebra is defined by the identity :
\begin{equation}\label{right1}
\displaystyle \big[[a , b],c  \big] =\big[[a,c] , b \big] +   
\big[a,[b,c]   \big],      
\end{equation}
which means for each $a \in A$, the map  $  x\mapsto  [x,a ]$  is a derivation of $A$.
A left or right Leibniz algebra in which bracket $[\cdot,\cdot]$ is skew-symmetric (alternating, if  $\K$ is of characteristic $2$) is a Lie algebra.
Notice
from (1) that a left Leibniz algebra $A$ satisfies,	for all $a,x,y\in A$
$ [a,[ x,y]]=-[a,[ y,x]], $  
and dually, a right Leibniz algebra satisfies   $ [[ x,y],a]=-[[ y,x],a]$. Symmetric Leibniz algebras satisfy
\begin{equation}\label{symmetric1}
[a,[ x,y]]=-[[ x,y],a].     
\end{equation}

In
the last years, the theory of Leibniz algebras has been extensively studied. Many
results on Lie algebras have been generalized to the case of Leibniz
algebras (\cite{CassasClssif,covez,cuvier,Demir,Fialowski,Geoffrey,Gorbatsevich,Gomez,Ray}). More recently, several papers deal with a so-called symmetric Leibniz algebras, which are  left and et right Leibniz algebras.

A superalgebra is a $\Z_2$-graded algebra $A=A_{0}\oplus A_{1}$  (that is, if $a\in A_\alpha,\, b\in A_\beta$,  $\alpha ,\,\beta\in \Z_2=\{0,1\}$, then, $ab\in A_{\alpha+\beta}$). The elements of $ A_{\alpha} $ are said to be homogeneous and of parity $ \alpha $.
A Lie superalgebra
is a superalgebra $A$ with an operation $[\cdot,\cdot]$ satisfying the following
identities:
\begin{align*}
[a,b]&=(-1)^{|a||b|} [b,a]\qquad &\text{  (Skew-supersymmetry)}\\
\displaystyle  [a,[ b,c ]  ]&=[[a , b],c   ] + (-1)^{|a||b|} 
\big[b ,[a,c]   \big] \qquad &\text{ (Super Jacobi identity)  }
\end{align*}
for all homogeneous elements $a,b,c\in A_{0}\cup A_{1}$.\\
Leibniz superalgebras appeared as an extension of Leibniz algebras (see \cite{Camacho,Hu}), in a similar way than Lie superalgebras generalize Lie algebras, motivated
in part for its applications in Physics. 
Color Lie (super)algebras, originally introduced in  \cite{Ritt,RittV}, can be seen as a direct generalization
of Lie (super)algebras.
Indeed, the latter are defined through antisymmetric (commutator) or
symmetric (anticommutator) products, although for the former  product is neither symmetric
nor antisymmetric and is defined by means of a commutation factor. This commutation factor
is equal to $\mp 1$ for (super)Lie algebras and more general for arbitrary color Lie (super)algebras.
As happened for Lie superalgebras, the basic tool to define color Lie (super)algebras is a grading
determined by an abelian group. The latter, besides defining the underlying grading in the
structure, moreover, provides a new object known as commutation factor.

Hom-algebra structures are given on linear spaces by products twisted by linear
maps. Hom-Lie algebras and general quasi-Hom-Lie and quasi-Lie algebras were introduced
by Hartwig, Larsson and Silvestrov as algebras embracing Lie algebras, super and color
Lie algebras and their quasi-deformations by twisted derivations.
In \cite{homtype}, the authors gives a systematic exploration of other possibilities to define Hom-type algebras.
In \cite{MakhloufABihom}, the authors introduced a generalized algebraic structure endowed with two commuting
multiplicative linear maps, called BiHom-algebras.These algebraic structures include
BiHom-associative algebras,   BiHom-Lie algebras and BiHom-Leibniz algebras.  In some particular cases, when  the two linear maps are
the same,  BiHom-algebras led to Hom-algebras. \\
In this paper, we study  symmetric  (super)Hom-Leibniz, (colour)BiHom-Leibniz, BiHom-Leibniz algebras of type $B_1$ and $B_2$.
Recall from \cite{MakhloufABihom}, that  a BiHom-Lie algebra is 
a $4$-tuple    $ (L,[\cdot,\cdot],\alpha,\beta) $ where $ L $ is a $ \K $-linear
space, $ \alpha,  $  $ \beta\, \colon L\to  L $    are linear maps and   $[\cdot,\cdot] \colon L\times L\to  L  $
is a bilinear map, satisfying the
following conditions, for all $ x,\, y,\, z\in L $ :
\begin{gather}
\alpha\circ\beta=\beta\circ\alpha\\
\alpha\left(  [x,y] \right) =\left[ \alpha(x),\alpha(y) \right] \text{\quad and  \quad  } 
\beta\left(  [x,y] \right) =\left[ \beta(x),\beta(y) \right] \\
\left[ \beta(x),\alpha(y)\right] =-\left[ \beta(y),\alpha(x) \right], \text{ \quad  (BiHom-skew-symmetry) \quad }\label{BiHom-skew}\\
\displaystyle \circlearrowleft_{x,y,z} \big[ \beta^2(x) ,\left[ \beta(y),\alpha(z) \right]  \big]=0. \text{ \quad  (BiHom-Jacobi condition)     \quad } \label{BiHom-Jacobi}
\end{gather}

Obviously, there is a hierarchy of algebras.
\begin{equation}\label{hierarchy}
\{ \text{   left Leibniz}  \}\cup\{ \text{   right  Leibniz}  \}\supseteq \{ \text{   symmetric Leibniz}  \}\supseteq \{ \text{   Lie}  \}
\end{equation}
and
\begin{equation}\label{hierarchyLie}
\displaystyle\{ \text{   Hom-Lie}  \}\supseteq_{\alpha=id_L} \{ \text{   Lie}  \}.
\end{equation}

Hence, one recovers Lie algebras properties  from Hom-Lie algebras when stating  $\alpha=id_L$. It turns out that the hierarchy \eqref{hierarchyLie} doesn't hold for BiHom-Lie algebras.
This led us to construct special types  of BiHom-algebras  that retains this property.
We introduce Bihom-Lie algebras of type $B_1$, where  \eqref{BiHom-skew} and \eqref{BiHom-Jacobi} are replaced by 
$ \left[ \beta(x),\beta^2(y)\right] =-\left[ \beta(y),\beta^2(x) \right]$ and  $\displaystyle \circlearrowleft_{x,y,z} \big[ \alpha(x) ,\left[ \beta(y),\beta^2(z) \right]  \big]=0 $.
Then 
\begin{equation*}
\displaystyle \{ \text{   BiHom-Lie type $B_1$}  \}\supseteq_{\beta=id}   \{ \text{   Hom-Lie}  \}\supseteq_{\alpha=id} \{ \text{   Lie}  \}.
\end{equation*}
A BiHom-Lie algebra in the usual sense should be referred to be "BiHom-Lie algebra of type $B_2$".
With  BiHom-Lie algebras of type $B_2$, we have
\begin{equation*}
\displaystyle   \{ \text{   Hom-Lie type $I_2$}  \}\supseteq_{\alpha=id} \{ \text{   Lie}  \}  \text{   and }
\displaystyle \{ \text{   BiHom-Lie type $B_2$}  \} \supseteq_{\substack{\alpha=id\\ \beta=id_L}} \{ \text{   Lie}  \},                     
\end{equation*}
where, the Hom-Lie algebra of type $I_2$ is given by the  skew-symmetric       bilinear bracket  satisfying $\displaystyle \circlearrowleft_{x,y,z} [x,[ y,\alpha(z) ]=0
$.\\


Throughout the article, we mean by a (Bi)Hom-Leibniz algebra a left or right or symmetric (Bi)Hom-Leibniz algebra.

The paper is organized as follows. In Section $1$, we recall definitions and some
key constructions of Hom-Leibniz (super)algebras and give some example of symmetric Hom-Leibniz (super)algebras. Moreover we introduce the concept of centroid and quasicentroid  for  Hom-Leibniz superalgebra (left or right or symmetric) and study some of their properties. In Section $2$, we give some constructions of Hom-Lie algebras by BiHom-Lie algebras and  conversely.
Also, we provide a construction of
BiHom-Leibniz algebras  $L_{(\alpha,\beta)} =(L,\{\cdot,\cdot  \},\alpha, \beta) $ from a Leibniz algebras $(L,[\cdot,\cdot])$. 
We also give some basic  definitions, properties of Ideals of BiHom-Leibniz algebras. This section also includes the concept  of generalized derivations of a BiHom-Leibniz algebras  and some   properties.
Section $3$ is dedicated to BiHom-Leibniz colour algebras, we give  a BiHom-Lie colour algebra   $ (\mathcal{A},[\cdot,\cdot],\alpha,\beta,\varepsilon) $  from an associative colour  algebra $ (\mathcal{A},\mu,\varepsilon) $.  We construct color BiHom-Leibniz algebras starting from two even centroids of  Leibniz or BiHom-Leibniz colour  algebras. In Section $4$, we define BiHom-Lie and BiHom-Leibniz algebras of type $B_1$, we study representations of symmetric BiHom-Leibniz algebras of type $B_1$, and give their cohomology. We show
that any extension of a symmetric BiHom-Leibniz algebras of type  $B_1$, are controlled by the second cohomology with respect its corresponding representation. 

\section{ Symmetric Hom-Leibniz (super)algebras}
In this section, we define the symmetric Hom-Leibniz (super)algebras generalizing the well known Leibniz  algebras given in \cite{Samiha,Geoffrey} and we gives a few examples of the symmetric Hom-Leibniz (super)algebras. We also study some properties of centroids of Hom-Leibniz superalgebras.
\subsection{ Symmetric Hom-Leibniz algebras    }
\begin{defn}\label{HOM}\cite{HLS,Ayedi,AbdenacerM}
	A Hom-Lie algebra  is a triple $(\mathcal{G},\ [\cdot,\cdot],\ \alpha)$\ consisting of a $ \mathbb{K} $ vector space $\mathcal{G}$, a  bilinear map \ $\ [\cdot,\cdot]:\mathcal{G}\times \mathcal{G}\rightarrow \mathcal{G}$ \ and a $ \mathbb{K}$-linear map\ $ \alpha:\mathcal{G}\rightarrow \mathcal{G} \ $satisfying
	\begin{eqnarray}
	&[x,y]=-[y,x],  \qquad &\text{(skew-symmetry) }\\
	&\circlearrowleft_{x,y,z}\big[\alpha(x),[y, z]\big]= 0 ,&\text{(Hom–Jacobi identity)}\label{jacobie} 
	\end{eqnarray}
	for all $x,\, y,\, z \in \mathcal{G}$. The two condition lead to the following identities. 
	\begin{gather}
	\left[\alpha(x),[y,z]\right]=	\left[ [x,y],\alpha(z)\right] +  \left[\alpha(y), [x,z]\right],\label{leftLeibniz}\\
	\left[\alpha(x),[y,z]\right]=	\left[ [x,y],\alpha(z)\right] - \left[ [x,z],\alpha(y)\right]\label{rigtLeibniz} 
	\end{gather}   for all $  x,\ y,\ z\in \mathcal{G}.$
\end{defn}
Now, we define  left and right Hom-Leibniz algebras. 
\begin{defn}
	
	A left (resp. right) Hom-Leibniz algebra is a $\K$-vector space $L$ equipped with a bracket operation $[\cdot,\cdot]$ and
	a linear map $\alpha$ that satisfy the equation \eqref{leftLeibniz} (resp. \eqref{rigtLeibniz}).
\end{defn}
Obviously, a Hom-Lie algebra is a  left and right Hom-Leibniz algebra. If $\alpha=id_L$, 
then  a left (resp. right) Hom-Leibniz algebra becomes  a left (resp. right) Leibniz algebra. A left (resp. right) Hom-Leibniz algebra is a Hom-Lie algebra if and only if $ [x,x]=0$, $\forall x\in L$.

\begin{defn}
	A  triple $(L,[\cdot,\cdot],\alpha)$ is called a symmetric
	Hom-Leibniz algebra if it is a left and a right Hom-Leibniz algebra.
\end{defn}
\begin{example}
	Let $ (e_1,e_2) $ be 	a basis of vector space $ L $  and $ \begin{pmatrix}
	-1&0\\
	0&1      
	\end{pmatrix} $ a matrix of linear map $\alpha$  with respect to this basis.    
	In the following table we give all possible cases for $ [\cdot,\cdot] $ to be a left Hom-Leibniz algebra\\
	
	
\footnotesize{
	\begin{tabular}{|c|c|}
		\hline
		Left Hom-Leibniz bracket & Remark  \\
		\hline
		$[e_1,e_1]=x e_2, \,	[e_1,e_2]=y e_2 ,\, [e_2,e_1]=[e_2,e_2]=0$ &      $ L$ is  multiplicative $\iff$  $y=0 $        \\
		&$ L$ is  symmetric $\iff$  $xy=0 $  \\
		\cline{2-1}
		\hline
		
		
		$[e_1,e_1]=[e_1,e_2]=0$,$\,	[e_2,e_1]=c e_1,\,  
		[e_2,e_2]=d e_1 \,$
		&      $ L$ is  multiplicative $\iff$  $c=0 $               \\	
			&$ L$ is  symmetric $\iff$  $c=0 $  \\
		\cline{2-1}	
		\hline
		
		$[e_1,e_1]=ae_1+xe_2, \,
		[e_1,e_2]=[e_2,e_1]=-\frac{  a}{x}[e_1,e_1]
		\,$,&$ L$ is  multiplicative $\iff$  $a=0 $\\ $\,
		[e_2,e_2]=\left( \frac{  a}{x}\right) ^{2}[e_1,e_1]
		\,$	
		&$ L$ is  symmetric \\
			\cline{2-1}	
		\hline
		
		
		
		
		\hline
		
		
		
		
		
		
		
		$[e_1,e_1]=[e_2,e_2]=0,\,      [e_1,e_2]=-[e_2,e_1]=be_1+y  e_2$	&$ L$ is  multiplicative $\iff$  $y=0 $  \\

		 & $ L$ is  a  Hom-Lie algebra\\
			\cline{2-1}
		\hline		
	\end{tabular}}      	            
\end{example}

\begin{example}
	Let $ (x_1,x_2,x_3) $ be a basis of $3$-dimensional space $  \mathcal{G}$ over $ \K $.
	Define a bilinear bracket operation on $\mathcal{G}\otimes \mathcal{G}  $ by 
	\begin{align*}
	[x_1\otimes x_3,x_1\otimes x_3]&=  x_1\otimes x_1\\	[x_2\otimes x_3,x_1\otimes x_3]&=  x_2\otimes x_1\\
	[x_2\otimes x_3,x_2\otimes x_3]&=  x_2\otimes x_2.
	\end{align*}
	The others brackets  are equal to $ 0 $. For any linear map $\alpha$ on $ \mathcal{G} $, 
	the triple $\left( \mathcal{G}\otimes \mathcal{G},[\cdot,\cdot],\alpha\otimes\alpha\right)   $ is not 
	a Hom-Lie algebra	but it is 
	a   symmetric Hom-Leibniz algebra.
\end{example}
In the following examples, we construct Hom-Leibniz algebras on a vector space $ L\otimes L$ starting from a Lie  or a Hom-Lie algebra $ L$. 


\begin{prop}\cite{TensorPowe}\label{tensorLie    }
	For any Lie algebra $(\mathcal{G}, [\cdot,\cdot])$, the bracket	 \[[x\otimes y,a\otimes b]=[x,[a,b]]\otimes y+x\otimes[y,[a,b]]\] defines a Leibniz algebra structure on the vector space $\mathcal{G}\otimes\mathcal{G}$.     
\end{prop}
\begin{prop}\cite{Yau}\label{tensorLeibn    }
	Let $(\mathcal{L},[\cdot,\cdot] )$ be a Leibniz algebra and $ \alpha : \mathcal{L}\to  \mathcal{L}$  be a Leibniz algebra endomorphism. Then  $(\mathcal{L},[\cdot,\cdot]_{\alpha},\alpha )$ is a Hom-Leibniz algebra, where $[x,y]_{\alpha}=[\alpha(x),\alpha(y)]  $.       
\end{prop}
Using \textbf{Proposition } \ref{tensorLeibn    } and  \textbf{Proposition } \ref{tensorLie    }, we obtain the following result.
\begin{prop}
	Let $(\mathcal{G},[\cdot,\cdot]') $ be a Lie algebra and $ \alpha : \mathcal{G}\to  \mathcal{G}$  be a Lie algebra endomorphism. We define on $\mathcal{G}\otimes\mathcal{G}$ the following bracket
	\[ \left[ x\otimes y,a\otimes b\right] =\left[ \alpha(x),[\alpha(a),\alpha(b)]']'\otimes \alpha(y)+\alpha(x)\otimes[\alpha(y),[\alpha(a),\alpha(b)]'\right]'  \] 	on $\mathcal{G}\otimes \mathcal{G}.$ Then
	$ \left(\mathcal{G}\otimes \mathcal{G},[\cdot,\cdot],\alpha \right)  $
	is a right  Hom-Leibniz algebra. 
\end{prop}
\subsection{ Centroids and derivations of  Hom-Leibniz superalgebra    } 
The concept of centroids and derivation of Leibniz algebras  is introduced in  \cite{centroiide}. Left  Leibniz superalgebras, originally were
introduced by Dzhumadil’daev in \cite{colourSuper}, can be seen as a direct generalization of Leibniz
algebras. The left Hom-Leibniz superalgebras  is introduced in \cite{HomSuperLeibniz}.
In this section, we introduce the notion of right and symmetric Hom-Leibniz superalgebras. Moreover, we introduce the  concept of centroids and derivation of Hom-Leibniz superalgebras .\\

Let  $V$ be a vector superspace over a field $\mathbb{K}$ that is a $\mathbb{Z}_{2}$-graded vector  space with a direct sum $V=V_{0}\oplus V_{1}.$  The elements of $V_{j}$, $j\in \mathbb{Z}_{2},$ are said to be homogenous  of parity $j.$ The parity of  a homogeneous element $x$ is denoted by $|x|.$ 
The space $End (V)$ is $\mathbb{Z}_{2}$-graded with a direct sum $End (V)=(End (V))_{0}\oplus(End (V))_{1}$ where
$(End (V))_{j}=\{f\in End (V) \mid f (V_{i})\subset V_{i+j}\}.$
The elements of $(End (V))_{j}$  are said to be homogenous of parity $j.$
\begin{defn}(see\cite{Ammar,Najib}    )
	A Hom-Lie superalgebra  is a triple $(\mathcal{G},\ [\cdot,\cdot],\ \alpha)$\ consisting of a superspace $\mathcal{G}$, an even bilinear map \ $\ [\cdot,\cdot]:\mathcal{G}\times \mathcal{G}\rightarrow \mathcal{G}$ \ and an even superspace homomorphism \ $ \alpha:\mathcal{G}\rightarrow \mathcal{G} \ $satisfying
	\begin{eqnarray*}
		&&[x,y]=-(-1)^{|x||y|}[y,x],\\
		&&\circlearrowleft_{x,y,z}(-1)^{|x||z|}\big[\alpha(x),[y, z]\big]= 0 \label{jacobie}
	\end{eqnarray*}
	for all homogeneous element x, y, z in $\mathcal{G}$.
\end{defn}
\begin{prop}
	Let $L$ be a   superspace
	and define the following vector subspace  $\Omega $  of $ End(L) $ consisting
	of linear maps $u$ on $L$ as follows:	
	\[ \Omega =\{ u\in End(L)\mid u\circ\alpha=\alpha \circ u    \}.\]	
	and the map \[\tilde{\alpha} \colon  \Omega  \to     \Omega;\quad \tilde{\alpha}(u) =\alpha \circ u.       \]	
	Then $ (\Omega,[\cdot,\cdot ]' ) $ (resp. $ (\Omega,[\cdot,\cdot ]', \tilde{\alpha})  $ ) is a Lie (resp. Hom-Lie) superalgebra with the bracket $ [u,v]'=uv-(-1)^{|u||v|} vu $ for all $ u,\, v \in \Omega. $	
\end{prop}
\begin{remq}
	  The Hom-Lie algebras $ (\Omega,[\cdot,\cdot ]', \tilde{\alpha})  $ is not necessarily multiplicative.            
\end{remq}
\begin{defn}(see\cite{Najib}    )
Let $(\mathcal{G},[\cdot,\cdot],\alpha)$ be a Hom-Lie superalgebra and $V=V_{0}\oplus V_{1}$ be an arbitrary vector superspace. Let $\beta\in\mathcal{G}l(V)$ be an arbitrary even linear self-map on $V$  and
$[\cdot,\cdot]_{V}\colon \mathcal{G}\times V \to V$
an even bilinear map.  
The triple $(V,[\cdot,\cdot]_{V}, \beta)$  is called a 
module on the Hom-Lie superalgebra $\mathcal{G}=\mathcal{G}_{0}\oplus \mathcal{G}_{1}$   if the  even bilinear map $[\cdot,\cdot]_{V}$ satisfies
\begin{eqnarray}
\left[[x,y],\beta(v)\right]_{V}&=&\left[\alpha(x),[y,v]_{V}\right]_{V}-(-1)^{|x||y|}\left[\alpha(y),[x,v]_{V}\right]_{V} \label{Rep2}
\label{mod}
\end{eqnarray}
for all homogeneous elements $x, y \in \mathcal{G}$ and $v\in  V .$\\

\end{defn}
\begin{defn}
	Let  $(\mathcal{L},[\cdot,\cdot],\alpha) $ be a triple consisting of a  superspace $ \mathcal{L} $, an even  bilinear map $ [\cdot,\cdot]\, :\mathcal{L}\times  \mathcal{L}\rightarrow  \mathcal{L}  $  and an even   superspace homomorphism $ \alpha : \mathcal{L}\rightarrow  \mathcal{L}  $. Then,
	$(\mathcal{L},[\cdot,\cdot],\alpha) $ is called :	
	\begin{enumerate}[(i)]
		\item a  left Hom-Leibniz superalgebra if it satisfies 
		\begin{equation*}
		\left[\alpha(x),[y,z]\right]=	\left[ [x,y],\alpha(z)\right] + (-1)^{|x||y|} \left[\alpha(y), [x,z]\right].
		\end{equation*}
		\item 	a right Hom-Leibniz superalgebra  if it satisfies 
		\begin{equation*}
		\left[\alpha(x),[y,z]\right]=	\left[ [x,y],\alpha(z)\right] -(-1)^{|y||z|} \left[ [x,z],\alpha(y)\right] \quad \forall\ x,\ y,\ z\in \mathcal{L}_{0}\cup \mathcal{L}_1.
		\end{equation*}
	\end{enumerate}
\end{defn}
The following proposition  provides a method to construct a left  Hom-Leibniz superalgebra by a   module of Hom-Lie superalgebra. 
\begin{prop}
	Let $(\mathcal{L},[\cdot,\cdot],\alpha) $ be a Hom-Lie superalgebra
	and  $(V,[\cdot,\cdot]_V,\beta) $ 
	a $\mathcal{L}$-module. Let $\varphi\colon V\to \mathcal{L}$  be an even linear map  satisfying 
	$\varphi([x,v]_{V})=[x,\varphi(v)]$
	  and $\varphi\circ \beta=\alpha\circ\varphi$. Then one can define a 
	 left  Hom-Leibniz superalgebra $(V,[\cdot,\cdot]',\beta) $ as follows: $[u,v]'=[\varphi(u),v]_V$. 
\end{prop}

Now, we define the symmetric	Hom-Leibniz superalgebra.
\begin{defn}
	If $(\mathcal{L},[\cdot,\cdot],\alpha) $  is a left and a right Leibniz algebra, then $\mathcal{L}$ is called a symmetric
	 Leibniz algebra.                       
\end{defn}

\begin{prop}
	A triple
	$(\mathcal{L},[\cdot,\cdot],\alpha) $ is a symmetric
	Hom-Leibniz superalgebra  if and only if 
	\begin{align*}
	\left[\alpha(x),[y,z]\right]&=	\left[ [x,y],\alpha(z)\right] + (-1)^{|x||y|} \left[\alpha(y), [x,z]\right];\\
	\left[\alpha(y), [x,z]\right]&=- (-1)^{\left( |x|+|z|\right) |y|} \left[ [x,z],\alpha(y)\right]    
	\end{align*}
	for all $ \ x,\ y,\ z\in \mathcal{L}_{0}\cup \mathcal{L}_1. $
\end{prop}
\begin{example}
	Let $\mathcal{L}=\mathcal{L}_{0}\oplus \mathcal{L}_{1}$  be a $3$-dimensional superspace, where $\mathcal{L}_{0}$ is generated by $e_1$, $e_2$ and $\mathcal{L}_{1}$ is generated by $e_3$. The product is given by
	\begin{gather*}
	[e_1,e_1] =ae_1+xe_2;\,
	[e_1,e_2]=[e_2,e_1]=-\frac{a}{x}[e_1,e_1];\,
	[e_2,e_2]=\left(\frac{a}{x}\right)^{2}[e_1,e_1];\\
	[e_3,e_3] =\frac{d}{x}  [e_1,e_1];\,[e_1,e_3] =[e_3,e_1]=[e_3,e_2] =    [e_2,e_3] =0  .
	\end{gather*}
	We consider the homomorphism  $\alpha \colon \mathcal{L}\to  \mathcal{L} $ defined by the matrix $ \begin{pmatrix}
	-1&0&0\\
	0&1&0\\
	0&0&\mu
	\end{pmatrix} $,with respect to basis $ (e_1,e_2,e_3)  $. Then $(\mathcal{L},[\cdot,\cdot],\alpha) $ is a symmetric
	Hom-Leibniz superalgebra.
\end{example}

\begin{defn}
A $\alpha^k$-derivation of a  Hom-Leibniz superalgebra $(\mathcal{L},[\cdot,\cdot],\alpha) $ is a  homogeneous linear map  $D\in \Omega$ satisfying 
\[           D([x,y])=[D(x),\alpha^{k}(y)]+(-1)^{|D||x|}[\alpha^{k}(x),D(y)]                                \]
for all $ \ x,\ y,\ z\in \mathcal{L}_{0}\cup \mathcal{L}_1. $ The set of all derivation of a Hom-Leibniz superalgebra $\mathcal{L}$ is denoted by $\displaystyle Der(L)=\bigoplus_{k\geq0}Der_{\alpha^{k}}(L)  $.
\end{defn} 
\begin{prop}
	Let  $(\mathcal{L},[\cdot,\cdot ],\alpha)$ be a left  (resp. right) Hom-Leibniz   superalgebra. For any $ a\in L $ satisfying $ \alpha(a)=a, $ define $ ad_{k}(a)\in End(L) $ (resp.  $ Ad_{k}(a)\in End(L) $) 	
	\[ad_{k}(a)(x)=[a,\alpha^{k}(x)] \]
	respectively 
	\[   Ad_{k}(a)(x)=(-1)^{|a||x|}[\alpha^{k}(x),a]  , \quad \forall x\in L.    \]	
	Then $ ad_{k}(a) $ (resp. $ Ad_{k}(a) $) is an $ \alpha^{k+1} $-derivation of the 
	left (resp. right) Hom-Leibniz algebra $ L $.
\end{prop}
\begin{defn}
	Let  $(L,[\cdot,\cdot ],\alpha)$ be a Hom-Leibniz   superalgebra. Then the 	$\alpha^k$-centroid of $L$ denoted as $C_{\alpha^{k}}(L)$ is defined by \[C_{\alpha^{k}}(L)=\left\lbrace d\in  \Omega\mid  d([x,y])=[d(x),\alpha^{k}(y)] = (-1)^{|d||x|}[\alpha^{k}(x),d(y)],\,\forall  \ x,\ y\in L_{0}\cup L_1                  \right\rbrace . \]
Denote by $\displaystyle C(L)=\bigoplus_{k\geq0}C_{\alpha^{k}}(L) $ the centroid of $\mathcal{L}$.
\end{defn}

\begin{defn}
	Let  $(L,[\cdot,\cdot ],\alpha)$ be a Hom-Leibniz   superalgebra and $d\in End(\mathcal{L})$.
	Then $d$ is called a $\alpha^k$-central derivation, if $d\in \Omega$ and \[ d([x,y])=[d(x),\alpha^{k}(y)]=(-1)^{|d||x|}[\alpha^{k}(x),d(y)]=0.\] 
The set of all central derivation of $\mathcal{L}$ is denoted by $ \displaystyle ZDer(L)=\bigoplus_{k\geq0}ZDer_{\alpha^{k}}(L) $.	
\end{defn} 
In the following of this section we study the structure of the centroids and derivations of  Hom-Leibniz superalgebras.
\begin{lem}
	Let $L$ be a  Hom-Leibniz superalgebra. Let $d\in Der_{\alpha^{k}}(L)  $ and 
	$\Phi\in C_{\alpha^{l}}(L)  $ then 
	\begin{enumerate}[(i)]
		\item $\Phi\circ d$ is an $\alpha^{k+l}$-derivation of $L$.
		\item $[\Phi,d]$ is $\alpha^{k+l}$-centroid of $L$.
		\item $ d\circ \Phi$ is an $\alpha^{k+l}$-centroid if and only if $\Phi\circ d$ is a $\alpha^{k+l}$-central derivation.
		\item $ d\circ \Phi$ is an $\alpha^{k+l}$-derivation only if only $[d,\Phi ]$ is a $\alpha^{k+l}$-central derivation.
	\end{enumerate}
\end{lem}
\begin{thm}
	Let $L$ be a  Hom-Leibniz superalgebra. Then \[ZDer_{\alpha^{k}}(L) =Der_{\alpha^{k}}(L)\cap C_{\alpha^{k}}(L).\]
\end{thm}
\begin{proof}
	The proof of previous theorem is similar to the case Leibniz algebra given in \cite{centroiide}.                
\end{proof}

\begin{prop}
	Let  $(L,[\cdot,\cdot ],\alpha)$ be a   Hom-Leibniz superalgebra. Then the following
	statements hold:
	\begin{enumerate}[(i)]
		\item $ ad(L)\subseteq  Der(L) \subseteq \Omega$, where $ad(L)  $ denotes the superalgebra of inner derivations of $ L $. 
		\item  $ ad(L) $, 	$ Der(L) $  and $ C(L) $ are Lie ( Resp. Hom-Lie ) subsuperalgebras of $ (\Omega,[\cdot,\cdot ]' ) $ (resp. $(\Omega,[\cdot,\cdot ]', \tilde{\alpha})  $ ))
	\end{enumerate}
	
\end{prop}
%



\section{ BiHom-Leibniz algebras}	
In this section, we recall the notion of BiHom-Lie algebras and then give some relations between BiHom-Lie algebras and  Hom-Lie algebras.
Moreover, we introduce the notion of symmetric BiHom-Leibniz algebras.  We introduce the  definitions and give some properties related to ideals of BiHom–Leibniz algebras and we extend       the concept of $(\alpha,\beta,\gamma)$-derivations
 of Lie algebras  introduced in \cite{Petr}  to  BiHom–Leibniz case.
\subsection{ BiHom-Lie algebras      }
We  first recall the definition of Hom-Lie algebra of type $I_3$ and then we give some connections between them and BiHom-Lie algebras. .
\begin{defn}\cite{homtype}
	A Hom-Lie algebra of type $I_3$ is defined by replacing, in Definition\ref{HOM}, equation\eqref{jacobie} by 
	\begin{equation}
	\circlearrowleft_{x,y,z}\big[x,[y, \alpha(z)]\big]= 0.        
	\end{equation}         
\end{defn}
\begin{defn}\label{bihom}\cite{bihom,ShengBi}
	A BiHom-Lie algebra over $ \K $ is a $4$-tuple    $ (L,[\cdot,\cdot],\alpha,\beta) $ where $ L $ is a $ \K $-vector
	space, $ \alpha,  $  $ \beta\, \colon L\to  L $    are linear maps and   $[\cdot,\cdot] \colon L\times L\to  L  $
	is a bilinear map, satisfying the
	following conditions, for all $ x,\, y,\, z\in L $ :
	\begin{gather*}
	\alpha\circ\beta=\beta\circ\alpha;\\
	\left[ \beta(x),\alpha(y)\right] =-\left[ \beta(y),\alpha(x) \right] \text{ \quad  (BiHom-skew-symmetry)};\\
	\big[ \beta^2(x) ,\left[ \beta(y),\alpha(z) \right]  \big]
	+ \big[ \beta^2(y) ,\left[ \beta(z),\alpha(x) \right]  \big]	 
	+ \big[ \beta^2(z) ,\left[ \beta(x),\alpha(y) \right]  \big]	 
	=0\\ \text{(BiHom-Jacobi condition)}. 
	\end{gather*}
\end{defn}
	In particular, if  $\alpha $ and $\beta $ preserves the bracket, then we call $ (L,[\cdot,\cdot],\alpha,\beta) $ a multiplicative BiHom-Lie algebra.\\
	 We recover  Hom-Lie algebra when we have $\beta=id_L $ and $\beta$ is a bijection.\\
	 If $ (L,[\cdot,\cdot],\alpha) $ is a  Hom-Lie algebra of type $ I_3 $ and $\alpha$ is in the centroid of $L$ (i.e $\alpha\left(  [x,y] \right) =\left[ \alpha(x),y \right]$), then   $ (L,[\cdot,\cdot],\alpha,id_L) $ is a BiHom-Lie algebra.	

In the following of this subsection, we establish a connection between BiHom-Lie algebra and  (original) Hom-Lie algebra.
\begin{prop}
	The $4$-tuple    $ (L,[\cdot,\cdot],\alpha,\alpha) $
	is a BiHom-Lie algebra if and only if   the triple $(\alpha(L),[\cdot,\cdot],\alpha)  $  is  a  Hom-Lie algebra.
\end{prop}
\begin{prop}
	If $(L,[\cdot,\cdot],\alpha,\beta)$  is a  BiHom-Lie algebra and we define the map $[\cdot,\cdot]' \colon L\times L\to  L,  $   $ [x,y]'=[\beta(x),\alpha(y)],  $ for all $ x,y\in L, $ then  $(L,[\cdot,\cdot]',\alpha\beta)$ is a Hom-Lie algebra.
\end{prop}
\subsection{  BiHom-Leibniz algebras }
Inspired by \cite{Samiha},  the definition of  generalized derivations of Lie algebras (see \cite{George}) and the definition of twisted
derivations (see\cite{Daniel})
, we introduce the  concept of BiHom-Leibniz algebra .\\

Let    $ (L,[\cdot,\cdot],\alpha,\beta) $ be a $4$-tuple consisting of a vector space $L$, a bilinear map 
	$[\cdot,\cdot] \colon L\times L\to  L  $ 	and  two linear maps $ \alpha,  $  $ \beta\, \colon L\to  L $ such that 
$\alpha\circ\beta=\beta\circ\alpha$,	$\alpha\left(  [x,y] \right) =\left[ \alpha(x),\alpha(y) \right]$ and 
$  \beta\left(  [x,y] \right) =\left[ \beta(x),\beta(y) \right] $ . Let $\Delta_{k,l}(L)$ denote the set of triples $(f,f',f")$ with $f,\,f',\,f"\in End(L)$
such that $[f(x),\alpha^{k}\beta^{l}(y)] +[\alpha^{k}\beta^{l}(x),f'(y)]=f"([x,y] ).        $ For all $a\in L$ define the endomorphisms $L_a,$ $R_a$ of $L$ by $ L_a(x)=[a,x]$,  $R_a(x)=[x,a]$.
\begin{defn}
\begin{enumerate}[(i)]
	\item If $(L_{\beta(x)},L_{\alpha(x)},L_{\beta\alpha(x)})\in\Delta_{0,1}(L), $ $\forall x\in L$, i.e \[  \big[\alpha\beta(x) ,[ y,z ]  \big]=\big[[\beta(x) , y],\beta(z)   \big] +   
	\big[\beta(y) ,[\alpha(x),z]   \big], \forall x,y,z\in L,	\]
	then $ (L,[\cdot,\cdot],\alpha,\beta) $ is called a  left  BiHom-Leibniz algebra. 
	\item 	If $(R_{\beta(z)},R_{\alpha(z)},R_{\beta\alpha(z)})\in\Delta_{1,0}(L), $ $\forall z\in L$, i.e
	\[  \displaystyle \big[[x , y],\alpha\beta(z)   \big] =\big[[x,\beta(z)] , \alpha(y)  \big] +   
	\big[\alpha(x),[y,\alpha(z)]   \big], \forall x,\, y,\, z \in L,
	\]then $ (L,[\cdot,\cdot],\alpha,\beta) $ is called a   right BiHom-Leibniz algebra.	 
\end{enumerate}		   
\end{defn}
\begin{prop}
	Let     $ (L,[\cdot,\cdot],\alpha,\beta) $ be a  left (respectively right) BiHom-Leibniz algebra. Then  
	\[  \displaystyle  \big[ [ \beta(x),\alpha(y) ], \alpha\beta(z) \big]=-\big[[\beta(y) , \alpha(x)],\alpha\beta(z)   \big],
	\]
	respectively
	\[  \displaystyle \big[\alpha\beta(x),\left[ \beta(y) , \alpha(z)\right]   \big] =-  
	\big[\alpha\beta(x),\left[ \beta(z),\alpha(y)\right]    \big],
	\]	
	for all  $x,\, y,\, z \in L$.      
\end{prop}

\begin{defn}
	If $(L,[\cdot,\cdot],\alpha,\beta) $ is a left and a right 	BiHom-Leibniz algebra, then $L$ is called a symmetric
	BiHom-Leibniz algebra.	
\end{defn}
\begin{prop}\label{1thsymmetic }
	Let $(L,[\cdot,\cdot],\alpha,\beta)$  be  a left BiHom-Leibniz algebra. Then $(L,[\cdot,\cdot],\alpha,\beta)$  is a symmetric
	BiHom-Leibniz algebra if and only if
	\begin{equation}\label{thsymmetic }
	\Big[ \beta(y),\left[  \alpha(x) ,\alpha(z)   \right]      \Big] =-   \Big[\left[  \beta(x) ,\beta(z)   \right] ,\alpha (y)     \Big],  
	\end{equation}
	for all  $x,\, y,\, z \in L$. 	
\end{prop}
\begin{prop}
	If  $(L,[\cdot,\cdot])$  is a symmetric Leibniz algebra and 
	$\alpha,\, \beta \colon L\to  L$ are two commuting morphisms of Leibniz algebras, and we define the map  $\{\cdot,\cdot  \}  \colon L\times L\to  L $, $ \{x,y  \}=[\alpha(x),\beta (y)],  $
	for all $x,\, y\in L,$  then $(L,\{\cdot,\cdot  \},\alpha, \beta) $ is a symmetric BiHom-Leibniz algebra, called the Yau twist of $L$
	and denoted by $L_{(\alpha,\beta)} $.
\end{prop}
\begin{proof}
	Since 	$(L,[\cdot,\cdot])$  is a left Leibniz algebra, there $L_{(\alpha,\beta)} $ is a left BiHom-Leibniz algebra (see \cite{bihom}). It remains to show the equality \ref{thsymmetic } is satisfied:
	\begin{align*}
	\{\beta(x),\{\alpha(x) ,\alpha(y)  \}
	&=\{\beta(x),[\alpha^2(x) ,\alpha\beta(y)]  \}\\
	&=\left[\alpha\beta(x),[\beta\alpha^2(x) ,\alpha\beta^2(y)]\right].
	\end{align*}
	Since   $(L,[\cdot,\cdot])$  is a symmetric Leibniz algebra, we have
	\[  \left[ a,[b,c]\right]  =-\left[ [b,c],a\right]     \text{  (see \cite{Samiha})  }.     \]
	Therefore, we have 
	\begin{align*}
	\{\beta(x),\{\alpha(x) ,\alpha(y)  \}
	&=-\left[[\beta\alpha^2(x) ,\alpha\beta^2(y)],\alpha\beta(x)\right]\\
	&=-\lbrace [\beta\alpha(x) ,\beta^2(y)],\alpha(x)\rbrace \\
	&=-\lbrace \{\beta(x) ,\beta(y)\},\alpha(x)\rbrace .
	\end{align*}
	By Proposition \ref{1thsymmetic }, we deduce that $L_{(\alpha,\beta)} $ is a symmetric BiHom-Leibniz algebra.
\end{proof}
The following results gives a way to construct Hom-Leibniz algebra starting from a BiHom-Leibniz algebra.
\begin{prop}
	Let $(L,[\cdot,\cdot],\alpha,\beta)$  be  a  BiHom-Leibniz algebra. Define the bilinear map  $\{\cdot,\cdot  \}  \colon L\times L\to  L $, $ \{x,y  \}=[\beta(x), \alpha(y)],  $
	for all $x,\, y\in L.$	Then $(L,\{\cdot,\cdot  \} ,\alpha\beta)$
	is a Hom-Leibniz algebra.	
\end{prop}

\begin{prop}	
	Let $L$ be a vector space,  $[\cdot,\cdot] \colon L\times L\to  L  $
	a bilinear map, $\alpha,\, \beta \colon L\to  L$ two commuting linear maps such that  
	$ 	\alpha\left(  [x,y] \right) =\left[ \alpha(x),\alpha(y) \right]$  and $
	\beta\left(  [x,y] \right) =\left[ \beta(x),\beta(y) \right] $,	
	for all  $x,\, y \in L$. 
	Define the bilinear map 
	$\{\cdot,\cdot  \}  \colon L\times L\to  L $,    $ \{x,y  \}=[\beta(x),\alpha(y)],  $ for all $x,\, y\in L.$ Then:\\
	$(\alpha(L),[\cdot,\cdot],\alpha,\beta)$ is a BiHom-Lie algebra  if and only if $(L,\{\cdot,\cdot  \} ,\alpha\beta)$ is a symmetric Hom-Leibniz algebra.
\end{prop}
\subsection{      Ideals of  BiHom-Leibniz algebras }
In this 
subsection, we extend Ideals of Hom-Leibniz algebras introduced in \cite{CassasJM} to BiHom-Leibniz algebras.
 
\begin{defn}
	Let $ (L,[\cdot,\cdot],\alpha,\beta) $ be a BiHom–Leibniz algebra.
	A  	 vector subspace $H$ of $L$ is called a BiHom–Leibniz subalgebra of $ (L,[\cdot,\cdot],\alpha,\beta) $ if $\alpha(H)\subset H, $ 
	$\beta(H)\subset H $ and $[H,H]\subset H$. In particular, a BiHom–Leibniz subalgebra $H$ is said to be a two-sided ideal if $[h,l],\, [l,h]\in H$ for all $l\in L,$ $h\in H.$ If only one relation holds, then we call $H$ a right (or left) ideal.\\	
	If $(H,[\cdot,\cdot], \alpha_H,\beta_H)$ is 
	a two-sided ideal, then the quotient $L/H$ is endowed with a BiHom–Leibniz algebra structure, naturally induced from the bracket on $L$.\\
	The commutator of two-sided ideals $H$ and $K$ of a BiHom–Leibniz
	algebra $L$, denoted by $[H,K]$, is the BiHom–Leibniz subalgebra of  $L$ spanned
	by the brackets  $[h,k]$ and $[k,h]$ for all $h\in H$, $k\in K.$	             
\end{defn}
The following lemma can be readily checked.
\begin{lem}
	Let $H$ and $K$ be two-sided ideals of a BiHom–Leibniz algebra $ (L,[\cdot,\cdot],\alpha,\beta) $.
	The following statements hold:
	\begin{enumerate}[(a)]
		\item  $H\cap K$ and $H+K$ are two-sided ideals of $L$;
		\item  $[H,K]\subseteq H\cap K$;	
		\item   $[H,K]$ is a two-sided ideal of $H$ and $K$. In particular, $[L,L]$
		is a two-sided ideal of $L$ ;
		\item $\alpha(L) $ and $\beta(L)$  are  BiHom–Leibniz subalgebras of $L$;
		\item  	If 	$L$ is a left (resp. right) BiHom–Leibniz algebra and $\beta$ (resp. $\alpha$ )  is surjective, then $[H,K]$ is an ideal of $L$;
		\item If $\alpha$ and $\beta$ are surjective, then $[H,K]$  is a two-sided ideal of $L$.
	\end{enumerate}	           
\end{lem}
In the following,
  we extend some result  related in ideals of  Leibniz algebras introduced in \cite{Samiha} to  BiHom-Leibniz case:
\begin{prop}
	Let $ (L,[\cdot,\cdot],\alpha,\beta) $ be a left (resp. right) BiHom–Leibniz algebra and  $ I_L$
	the vector
	space $ I_L$ spanned by the set $\{ [x,x] \mid \, x\in L                 \}$. If $\beta$  (resp. $\alpha$ ) is a surjective homomorphism, then,  $ I_L$ is an ideal of $L$. Moreover, $[I_L,L]=0$ (resp. $[L,I_L]=0$  ).\\
	It is clear that
	$L$ is a BiHom–Lie algebra if and only if $ I_L=\{0\}$. Therefore, the quotient algebra $L/ I_L$
	is a BiHom–Lie algebra.
\end{prop}
\begin{prop}
	If  $(L,[\cdot,\cdot],\alpha,\beta)$  is  a symmetric  BiHom-Leibniz algebra,   then the   two-sided ideal
	$\left( L^2=[L,L],[\cdot,\cdot]_{/L^2\times L^2},\alpha_{/L^2},\beta_{/L^2}\right) $ is  a BiHom–Lie algebra.       
\end{prop}
\begin{defn}
	Let $ (L,[\cdot,\cdot],\alpha,\beta) $ be a  BiHom–Leibniz algebra. The subspace $Z(L)=\{x\in L\mid [x,y]=0=[y,x],\forall y\in L \}$ of $L$ is said to be the center of $L$. Note that if $\alpha$ and 
	$\beta$ are surjective, then $Z(L)$ is a two-sided
	ideal of L.
\end{defn}
\subsection{$(\lambda,\mu,\gamma) $- 
	derivations }
Let $ (L,[\cdot,\cdot],\alpha,\beta) $ be a BiHom–Leibniz algebra.
We set
\[\Omega =\{f\in End(L)\mid f\circ \alpha=\alpha\circ f,f\circ\beta=\beta\circ f\}    .\]
\begin{defn}
	For any integer $ k,\, l $, a linear map $ D\colon L\to L $
	is called an
	$ \alpha^k\beta^l $-derivation of the BiHom-Leibniz  algebra $(L,[\cdot,\cdot ],\alpha,\beta)$,
	if
	$D\in \Omega$ and
	\begin{eqnarray*}
		D([x,y])&=&[D(x),\alpha^{k}\beta^{l}(y)]+[\alpha^{k}\beta^{l}(x),D(y)],
	\end{eqnarray*}
	for all  $ x,y \in L$.  Denote by $\displaystyle Der(L)=\bigoplus_{0\leq k,l} Der_{\alpha^k\beta^l}(L)$ the set of derivations of the Bihom-Leibniz algebra $ (L,[\cdot,\cdot],\alpha,\beta) $.
\end{defn}

\begin{prop}
	Let  $(L,[\cdot,\cdot ],\alpha,\beta)$ be a left  BiHom-Leibniz   algebra and $ a\in L $. 
	Define respectively $ D$, $D'$, $ D'' $
	by 
	\begin{equation*}
	D(x)=[\alpha\beta(a),\alpha^{k}\beta^{l}(x)],\quad  D'(x)=[\beta(a),\alpha^{k}\beta^{l}(x)],\quad   D''(x)=[\alpha(a),\alpha^{k}\beta^{l}(x)].  
	\end{equation*}	
	Then 
	\begin{enumerate}[(i)]
		\item \quad  $ \alpha\circ D'=D\circ\alpha, $ $ \beta\circ D''=D\circ\beta $;
		\item \quad$ 	[D(x),\alpha^{k}\beta^{l}(y)]+[\alpha^{k}\beta^{l}(x),D'(y)]=D''([x,y]) $;
		\item \quad If $ \alpha^{2}(a) =\alpha(a)=\beta(a)$, then $ D $ is an $\alpha^{k}\beta^{l+1}  $-derivation   of the 
		left  BiHom-Leibniz algebra $ L $.
	\end{enumerate}
\end{prop}
\begin{prop}
	Let  $(L,[\cdot,\cdot ],\alpha,\beta)$ be a symmetric  BiHom-Leibniz   algebra and $ a\in L $. 	For any $ a\in L $, define 
	respectively	
	$ ad_{kl}(a)$, $ Ad_{kl}(a)\in End(L) $ by 
	\[ad_{kl}(a)(x)=[a,\alpha^{k}\beta^{l}(x)],\quad  \quad Ad_{kl}(a)(x)=[\alpha^{k}\beta^{l}(x),a],\, \forall x\in L.  \] Then	
	\begin{enumerate}[(i)]
		\item
		 \begin{align*}
		   \alpha\circ ad_{kl}(a)=ad_{kl}(\alpha(a))\circ\alpha;& \quad&
		\beta\circ ad_{kl}(a)=ad_{kl}(\beta(a))\circ\beta; \\
		\alpha\circ Ad_{kl}(a)=Ad_{kl}(\alpha(a))\circ\alpha;&\quad&
		\beta\circ Ad_{kl}(a)=Ad_{kl}(\beta(a))\circ\beta;
	\end{align*}	
		\item 
		\begin{align*}
		Ad_{kl}(\alpha\beta(a))([x,y])&=-ad_{k+1,l-1}(\beta^2(a))([x,y])  ;    \\
		ad_{kl}(\alpha\beta(a))([x,y])&=-Ad_{k-1,l+1}(\alpha^2(a))([x,y]);\\
		Ad_{kl}(\alpha\beta(a))([x,y])&=[Ad_{kl}(\beta(a))(x),\alpha^{k+1}\beta^{l}(y)]-[\alpha^{k+1}\beta^{l}(x),ad_{k+1,l-1}(\beta(a))(x)].
		\end{align*}		  
		\item \quad If $ \alpha\beta(a) =\alpha(a)=\beta(a)$, $Ad_{k,l}  $ and $ad_{k,l}  $  are $\alpha^{k}\beta^{l+1}  $-derivation   of the 
		symmetric  BiHom-Leibniz algebra $ L $.
	\end{enumerate}
\end{prop}
\begin{defn} Let $ (L,[\cdot,\cdot],\alpha,\beta) $ be a BiHom–Leibniz algebra and $\lambda,\mu,\gamma $
	be elements of $\K$	.
	A linear map $d\in \Omega$ is a $ (\lambda,\mu,\gamma)$-$ \alpha^k\beta^l $-derivation of $L$ if for all $x,y\in L$ we have\[ \lambda\,	d([x,y])=\mu\, [d(x),\alpha^{k}\beta^{l}(y)]+ \gamma\,[\alpha^{k}\beta^{l}(x),d(y)]. \]
	
	We denote the set of all $ (\lambda,\mu,\gamma)$-derivations by $\displaystyle Der^{(\lambda,\mu,\gamma)}(L)=\bigoplus_{0\leq k,l} Der_{\alpha^k\beta^l}^{(\lambda,\mu,\gamma)}(L).$
\end{defn} 
\begin{defn}
	Let $ (L,[\cdot,\cdot],\alpha,\beta) $ be a BiHom–Leibniz algebra over a field $\K$ (char $\K\neq 2$) and $\lambda,\mu,\gamma\in \K $. 
	Then for $ Der_{\alpha^k\beta^l}^{(\lambda,\mu,\gamma)}(L)$
	we fix the followings particular cases: 
	\begin{enumerate}[(a)]
		\item $\displaystyle Der_{\alpha^k\beta^l}^{(1,1,1)}(L)=Der_{\alpha^k\beta^l } (L)    ,$ 	   
		\item $\displaystyle Der_{\alpha^k\beta^l}^{(1,1,0)}(L)=\{ d\in  \Omega\mid d([x,y]) = [d(x),\alpha^{k}\beta^{l}(y)]          \},$ 
		is called the $\alpha^k\beta^l$-left-centroid of $L$;	   
		\item $\displaystyle Der_{\alpha^k\beta^l}^{(1,0,1)}(L)=\{ d\in  \Omega \mid d([x,y]) = [\alpha^{k}\beta^{l}(x),d(y)]          \},$ 
		is called the $\alpha^k\beta^l$-right-centroid of $L$;
		\item $\displaystyle Der_{\alpha^k\beta^l}^{(1,1,0)}(L)\cap Der_{\alpha^k\beta^l}^{(1,0,1)}(L)$, 
		is called the $\alpha^k\beta^l$-centroid of $L$ and it is noted $C_{\alpha^k\beta^l}(L)$;
		\item \begin{gather*}
		 \displaystyle Der_{\alpha^k\beta^l}^{(0,1,-1)}(L)\cap Der_{\alpha^k\beta^l}^{(1,1,-1)}(L)\\ =\{ d\in   \Omega\mid d([x,y]) = 0 =         [d(x),\alpha^{k}\beta^{l}(y)]=[\alpha^{k}\beta^{l}(x),d(y)]          \}       
		\end{gather*}
		is called the $\alpha^k\beta^l$-central derivation of $L$ and it is noted $Zder_{\alpha^k\beta^l}(L)$; 
		\item $ Der_{\alpha^k\beta^l}^{(0,1,-1)}(L)=\{ d\in   \Omega\mid         [d(x),\alpha^{k}\beta^{l}(y)]=[\alpha^{k}\beta^{l}(x),d(y)]          \},$ 
		is called the $\alpha^k\beta^l$-quasi-centroid of $L$ and it is noted $QC_{\alpha^k\beta^l}(L)$. 	   	 		
	\end{enumerate}        
\end{defn}
Now, we consider the subspace  \[ IDer_{\alpha^k\beta^l}^{(\lambda,\mu,\gamma)}(L)=\{d\in  \Omega  \mid \lambda\,	d([x,u])=\mu\, [d(x),\alpha^{k}\beta^{l}(u)]+ \gamma\,[\alpha^{k}\beta^{l}(x),d(u)], \forall x\in L,u\in L^2   \}.\]
\begin{thm}
	Let $ (L,[\cdot,\cdot],\alpha,\beta) $ be a symmetric BiHom–Leibniz algebra 
	such that the maps $\alpha$ and $\beta$ are surjective. Let  $\lambda,\mu,\gamma $
	be elements of $\K$	.
	\begin{enumerate}[(a)]
\item If $\lambda\neq 0$ and $\mu^2\neq \gamma^2$. Then $\displaystyle IDer_{\alpha^k\beta^l}^{(\lambda,\mu,\gamma)}(L)=\displaystyle IDer_{\alpha^k\beta^l}^{(\frac{ \lambda}{\mu+\gamma},1,0)}(L)$.
\item If $\lambda\neq 0$, $\mu\neq 0$ and $\gamma=-\mu$. Then $\displaystyle IDer_{\alpha^k\beta^l}^{(\lambda,\mu,\gamma)}(L)=\displaystyle IDer_{\alpha^k\beta^l}^{(\frac{ 1}{2},1,0)}(L)$.
\item If $\lambda\neq 0$, $\mu=\gamma$ and $\mu\neq 0 $. Then $\displaystyle IDer_{\alpha^k\beta^l}^{(\lambda,\mu,\gamma)}(L)=\displaystyle IDer_{\alpha^k\beta^l}^{(\frac{ \lambda}{\mu},1,1)}(L)$.
\item If $\lambda\neq 0$, $\mu=\gamma=0$. Then $\displaystyle IDer_{\alpha^k\beta^l}^{(\lambda,\mu,\gamma)}(L)=\displaystyle IDer_{\alpha^k\beta^l}^{(1,0,0)}(L)$.
\item If $\lambda= 0$ and $\mu^2\neq \gamma^2$. Then $\displaystyle IDer_{\alpha^k\beta^l}^{(\lambda,\mu,\gamma)}(L)=\displaystyle IDer_{\alpha^k\beta^l}^{(0,1,0)}(L)$.
\item If $\lambda= 0$ and $\mu= \gamma$. Then $\displaystyle IDer_{\alpha^k\beta^l}^{(\lambda,\mu,\gamma)}(L)=\displaystyle IDer_{\alpha^k\beta^l}^{(0,1,1)}(L)$.
\item If $\lambda= 0$ and $\mu= -\gamma$. Then $\displaystyle IDer_{\alpha^k\beta^l}^{(\lambda,\mu,\gamma)}(L)=\displaystyle IDer_{\alpha^k\beta^l}^{(0,1,-1)}(L)$.
 	\end{enumerate}	                                     
\end{thm}
\begin{proof}
Let $x\in L,\, u=[y,z]\in L^2.$	Since $\alpha$ and $ \beta$ are surjectives there exists $a,\,v\in L $ so that $x=\beta(a) $, $u=\alpha(v)=[\alpha(y),\alpha(z)]$. It follows from \eqref{thsymmetic } that $[\beta(a),\alpha(v)]=- [\beta(v),\alpha(a)]   $.
The rest of proof is similar to the one of case of   Lie algebras  given in \cite{Petr}.              
\end{proof}
\begin{cor}
	Let $ (L,[\cdot,\cdot],\alpha,\beta) $ be a  BiHom–Lie algebra 
	such that the maps $\alpha$ and $\beta$ are surjective. 
	For any $\lambda,\mu,\gamma\in \K$ 
	there exists $\delta\in \K$ such that 
	the subspace $ Der_{\alpha^k\beta^l}^{(\lambda,\mu,\gamma)}(L)$
	is equal to one
	of the four following subspaces: 	
	\begin{multicols}{2} 
		\begin{enumerate}[(a)]
			\item $\displaystyle Der_{\alpha^k\beta^l}^{(\delta,0,0)}(L)$;
			\item $\displaystyle Der_{\alpha^k\beta^l}^{(\delta,1,1)}(L)$;
			\item $\displaystyle Der_{\alpha^k\beta^l}^{(\delta,1,-1)}(L)$;
			\item $\displaystyle Der_{\alpha^k\beta^l}^{(\delta,1,0)}(L)$.
		\end{enumerate}
	\end{multicols}                            
\end{cor}

\section{BiHom-Leibniz colour algebras}
Let $\Gamma$ be an abelian group. A vector space $\mathcal{A}$ is said to be $\Gamma$-graded, if there is a family $(\mathcal{A}_\gamma)_{\gamma\in \Gamma}$
of vector subspace of $\mathcal{A}$ such that
$\mathcal{A}=\displaystyle \oplus_{\gamma\in \Gamma}\mathcal{A}_\gamma$. 
An element $x \in \mathcal{A}$ is said to be homogeneous of degree $\gamma$ if $x\in \mathcal{A}_\gamma$. We denote by $ \mathcal{H}(\mathcal{A})$ the set of all the homogeneous elements of $ \mathcal{A}$.
\begin{defn}
	A map $ \varepsilon\colon \Gamma\times \Gamma\to \K^*$ 
	is called a
	skewsymmetric bicharacter on $\Gamma$
	if the following identities hold, for all 
	$a$, $b$, $c\in \Gamma$ 
	\begin{enumerate}[(i)]
		\item $\varepsilon(a,b)\varepsilon(b,a)=1$;
		\item 	$ \varepsilon(a,b+c)=\varepsilon(a,b)\varepsilon(a,c)$;
		\item   $\varepsilon(a+b,c)=\varepsilon(a,c)\varepsilon(b,c)$.
	\end{enumerate}
\end{defn}
If $x$ and $y$ are two homogeneous elements of degree $\gamma$ and $\gamma'$  respectively, then we
shorten the notation by writing $\varepsilon(x,y)$ instead of $\varepsilon(\gamma,\gamma')$.
\subsection{BiHom-Lie  colour  algebras       }
In the following we summarize definitions of BiHom–Lie and BiHom-associative
color algebraic structures  generalizing the well known Hom–Lie and Hom–associative
color algebras (See \cite{Yuan}  ).
\begin{defn}\label{ bihom}(see \cite{bihomC})
	A BiHom-Lie  colour algebra over a field $ \K $ is a $5$-tuple    $ (\mathcal{A},[\cdot,\cdot],\varepsilon,\alpha,\beta) $ consisting of a $\Gamma$-graded vector space $\mathcal{A}$, an even bilinear mapping   $[\cdot,\cdot] \colon \mathcal{A}\times \mathcal{A}\to   \mathcal{A} $ (i.e $[\mathcal{A}_{a},\mathcal{A}_{b}]\subset\mathcal{A} _{a+b} $ for all $a, b \in \Gamma$), a bicharacter $\varepsilon\colon \mathcal{A}\times  \mathcal{A}\to \K^*$ and two even homomorphisms $ \alpha,  $  $ \beta\, \colon \mathcal{A}\to  \mathcal{A} $ such that for all $ x,y,z\in \mathcal{H}(\mathcal{A})$ we have
	\begin{gather*}
	\alpha\circ\beta=\beta\circ\alpha\\
	\alpha\left(  [x,y] \right) =\left[ \alpha(x),\alpha(y) \right] \text{\quad and  \quad  } 
	\beta\left(  [x,y] \right) =\left[ \beta(x),\beta(y) \right] \\
	\left[ \beta(x),\alpha(y)\right] =-\varepsilon(x,y)\left[ \beta(y),\alpha(x) \right], \text{ \quad  ($\varepsilon$-BiHom-skew-symmetry) \quad }\\
	\displaystyle \circlearrowleft_{x,y,z} \varepsilon(z,x)\big[ \beta^2(x) ,\left[ \beta(y),\alpha(z) \right]  \big]=0. \text{ \quad  ($\varepsilon$-BiHom-Jacobi condition)     \quad } 
	\end{gather*}
\end{defn}
\begin{defn}[\textbf{BiHom-associative  colour  algebras}]\label{ bihom}(see \cite{bihomC})
	A BiHom-associative  colour algebra over $ \K $ is a $5$-tuple    $ (\mathcal{A},\mu,\varepsilon,\alpha,\beta) $ consisting of a $\Gamma$-graded vector space $\mathcal{A}$, an even bilinear mapping   $\mu \colon \mathcal{A}\times \mathcal{A}\to   \mathcal{A} $ (i.e $\mu({A}_{a},\mathcal{A}_{b})\subset\mathcal{A} _{a+b} $ for all $a, b \in \Gamma$), a bicharacter $\varepsilon\colon \mathcal{A}\times  \mathcal{A}\to \K^*$ and two even homomorphisms $ \alpha,  $  $ \beta\, \colon \mathcal{A}\to  \mathcal{A} $ such that
	$\alpha\circ\beta=\beta\circ\alpha$ and 
	for all $ x,y,z\in A$ we have 
	$\mu\left( \alpha(x),\mu(y,z)\right) =\mu\left( \mu(x,y),\beta(z)\right) .$\\
	In particular, if $\alpha(\mu(x,y))=\mu(\alpha(x),\alpha(y))$
	and $\beta(\mu(x,y))=\mu(\beta(x),\beta(y))$,
	we call it
	multiplicative Bihom-associative colour algebra. 	
\end{defn}
Inspired by \cite{bihom}, we give the following  constructions  of   BiHom-associative and  BiHom-Lie algebras starting by an ordinary associative colour algebra.
\begin{prop}
	Let $ (\mathcal{A},\mu,\varepsilon) $ 
	be an ordinary associative colour algebra and let 
	$ \alpha,  $  $ \beta\, \colon \mathcal{A}\to  \mathcal{A} $ 
	two
	commuting even linear maps such that 
	$	\alpha\left(  \mu(x,y) \right) =\mu\left(\alpha(x), \alpha(y)      \right) $  and  
	$\beta\left(  \mu(x,y)  \right) =\mu\left(  \beta(x),\beta(y) \right)$
	, for
	all $x,\, y\in  \mathcal{H}(\mathcal{A})$
	. Define the even linear map $\mu' \colon \mathcal{A}\times \mathcal{A}\to   \mathcal{A} ,$ 
	$   \mu'(x,y) =\mu\left( \alpha(x),\beta(y)  \right)$ 
	 (resp.
	$  [\cdot,\cdot] \colon \mathcal{A}\times \mathcal{A}\to   \mathcal{A},$ $[x,y] = \mu\left( \alpha(x),\beta(y)\right)  - \mu\left( \beta(y),\alpha(x)\right) $)
	\\
	Then $(\mathcal{A},\mu',\varepsilon,\alpha,\beta)$ (resp.  $\mathcal{A}_{\alpha,\beta}=(\mathcal{A},[\cdot,\cdot], \varepsilon,\alpha,\beta )$   )  is a BiHom-associative (resp. BiHom-Lie   ) colour algebra.
\end{prop}
\subsection{ BiHom-Leibniz colour algebras}
First we introduce a definition of  BiHom--Leibniz colour algebra.\\

We consider a $5$-tuple  $ (L,[\cdot,\cdot], \varepsilon,\alpha,\beta) $, 
consisting of a $\Gamma$-graded vector space
$ L $, an even bilinear map $[\cdot,\cdot] \colon L\times L\to  L  $ (i.e. $[L_a,L_b]\subset L_{a+b}$ for all $a,b\in \Gamma$), a bicharacter $\varepsilon\colon L\times L \to \K^*$ and two even homorphism $ \alpha,  $  $ \beta\, \colon L\to  L $  such that for all homogeneous elements $ x,y,z$
we have	
$\alpha\circ\beta=\beta\circ\alpha$, 
$\alpha\left(  [x,y] \right) =\left[ \alpha(x),\alpha(y) \right]$ and 
$  \beta\left(  [x,y] \right) =\left[ \beta(x),\beta(y) \right] .$

\begin{defn}
 The $5$-tuple   $ (L,[\cdot,\cdot], \varepsilon,\alpha,\beta) $ is said to be a
\begin{enumerate}[(i)]
	\item    left BiHom-Leibniz  colour
	algebra if for any homogeneous elements $ x,y,z\in L$  the so-called left Leibniz identity
		\[  \displaystyle  \big[\alpha\beta(x) ,[ y,z ]  \big]=\big[[\beta(x) , y],\beta(z)   \big] + \varepsilon(x,y)  
	\big[\beta(y) ,[\alpha(x),z]   \big]
	\] holds.
	\item	right BiHom-Leibniz  colour
	algebra if for any homogeneous elements $ x,y,z\in L$ if it satisfies  the identity
	\[  \displaystyle \big[[x , y],\alpha\beta(z)   \big] =\varepsilon(y,z)\big[[x,\beta(z)] , \alpha(y)  \big] +   
	\big[\alpha(x),[y,\alpha(z)]   \big].
	\]	
\end{enumerate}	
\end{defn}
\begin{remq}	
	  \begin{enumerate}
	  	\item If $\Gamma=\{1\}$ and $ \varepsilon(x,y)=1 $, for all $x,\, y\in L$ then
	  $ (L,[\cdot,\cdot], \varepsilon) $ is a 	   left  (resp. right) Leibniz  colour
	  algebra if and only if $ (L,[\cdot,\cdot],Id_L,Id_L, \varepsilon) $ is a 	   left  (resp. right) BiHom-Leibniz  colour
	  algebra.
	   	\item If $\Gamma=\Z_2$ and $ \varepsilon(x,y)=(-1)^{|x||y|} $, for all $x,\, y\in L$ then
	  $ (L,[\cdot,\cdot], \varepsilon) $ is a 	   left  (resp. right) Leibniz  
	  superalgebra if and only if $ (L,[\cdot,\cdot],Id_L,Id_L, \varepsilon) $ is a 	   left  (resp. right) BiHom-Leibniz  colour
	  algebra.
	   	\item If $\alpha$ is surjective  then
	$ (L,[\cdot,\cdot],\alpha, \varepsilon) $ is a 	   left  (resp. right) Hom-Leibniz  
	 colour  algebra if and only if $ (L,[\cdot,\cdot],\alpha,\alpha, \varepsilon) $ is a 	   left  (resp. right) BiHom-Leibniz  colour
	algebra.  
	  \end{enumerate}                                   
\end{remq}
\begin{prop}
	                   If  $ (L,[\cdot,\cdot], \varepsilon,\alpha,\beta) $ is a 	   left  (resp. right) BiHom-Leibniz  colour algebra, then \[	\big[\left[\beta(x) , \alpha(y)\right] ,\alpha\beta(z)\big]    \big]=-\varepsilon(x,y)\big[\left[ \beta(y),\alpha(x)\right]  , \alpha\beta(z)  \big]  , \qquad \forall x, y, z \in L.\]
	                 Respectively 	                  
	    \[	\big[\alpha\beta(z) ,\left[\beta(x), \alpha(y)\right]     \big]=-\varepsilon(x,y)\big[\alpha\beta(z),\left[ \beta(y),\alpha(x)\right]    \big]  , \qquad \forall x, y, z \in L,\]  
	for all homogeneous elements $ x, y, z \in L$.               
\end{prop}

\begin{defn}
	If $ (L,[\cdot,\cdot],\alpha, \varepsilon) $ is a left and a right BiHom-Leibniz  colour algebra, then $L$ is called a symmetric
	BiHom-Leibniz colour algebra.          
\end{defn}
\begin{prop}
	Let $ (L,[\cdot,\cdot], \varepsilon,\alpha,\beta) $ be a left BiHom-Leibniz  colour algebra. Then, $ (L,[\cdot,\cdot], \varepsilon,\alpha,\beta) $ is a  symmetric
	BiHom-
	Leibniz colour algebra if and only if \[	\big[\beta(y) ,\left[ \alpha(x),\alpha(z)\right]    \big]=-\varepsilon(y,x+z)\big[\left[ \beta(x),\beta(z)\right]  , \alpha(y)  \big]  , \qquad \forall x, y, z \in L.\]             
\end{prop}
In the following we construct BiHom–Leibniz colour algebras involving elements of the centroid of colour Leibniz algebras. Let
$ (L,[\cdot,\cdot], \varepsilon) $ be a   Leibniz  colour lalgebra.
An endomorphism $\alpha\in  End(L)_\gamma$  of degree $d$ is said to be an element of degree  $\gamma$ of the centroid if $\alpha([x,y])=[\alpha(x),y]=\varepsilon(\gamma,x)[x,\alpha(y)]  $ for all $x,y\in \mathcal{H}(L) $.
The
centroid of $L$ of degree $\gamma$  is defined by
\[ C_{\gamma}(L)=\{ \alpha\in    End(L)_\gamma\mid    \alpha([x,y])=[\alpha(x),y]=\varepsilon(\gamma,x)[x,\alpha(y)]              \}.                \]
\begin{prop}
	Let $ (L,[\cdot,\cdot],\varepsilon)$ be a
	Leibniz  colour algebra and  let $ \alpha, \beta \colon L\times L\to  L$ two commuting even linear maps such that $\alpha^2=\alpha$ and $\beta^2=\beta.$ Set for $ x,y\in \mathcal{H}(L)$,   $\{x,y\}=[\beta(x),\alpha(y)]$.  Then  $ (L,[\cdot,\cdot], \varepsilon,\alpha,\beta) $ is a
	BiHom-
	Leibniz  colour  algebra.
\end{prop}
\begin{proof}
	Since $\alpha$ and $\beta$ are even, then $ \varepsilon(\gamma,x)=1 $. Therefore, $\{x,y\}=[\beta(x),\alpha(y)]=\alpha\beta\left( [x,y] \right) $.
	Suppose that $ (L,[\cdot,\cdot],\varepsilon)$ is a  Leibniz  colour lalgebra. We have 
	\begin{equation*}
	\alpha \left( \{x,y\}\right)  
	= \alpha^2 \beta\left([x,y] \right) = \alpha^4\beta\left(  [x, y]\right)
	=\alpha^2\beta\left(  [ \alpha (x), \alpha(y)   ]\right)
	=\alpha\beta\left[ \alpha (x), \alpha(y) \right] 
	=\{ \alpha(x), \alpha(y)  \} .                
	\end{equation*} 
	Similarly, 	 \begin{equation*}
	\beta \left( \{x,y\}\right) 
	=\alpha \beta^2\left([x,y] \right) 
	= \alpha\beta^4\left(  [x, y]\right)
	=\alpha\beta^2\left(  [ \beta (x),\beta (y)   ]\right)
	=\alpha\beta\left(  [ \beta (x),\beta (y)   ]\right)
	=\{ \beta(x), \beta(y)  \} .  
	\end{equation*}  
	\begin{enumerate}
		\item 
		If   $ (L,[\cdot,\cdot],\varepsilon)$ is a left  Leibniz  colour algebra :
		\begin{equation*}
		\left\lbrace \alpha\beta(x),\{y,z \}    \right\rbrace
		= \left\lbrace  \alpha\beta(x), \alpha\beta\left(  [y,  z]   \right)            \right\rbrace
		= \alpha\beta\big( \left[  \alpha\beta(x) ,\alpha\beta\left(  [y,  z]   \right)    \right]\big)
		= \alpha^3\beta^3\big( \left[  x ,[y,  z]       \right]\big).	       	          
		\end{equation*} 
		\begin{align*}
		&	\left\lbrace \{\beta(x),y\},\beta(z) \}    \right\rbrace
		+\varepsilon(x,y)\left\lbrace  \beta(y),\{\alpha(x) ,z  \}       \right\rbrace \\
		&= \alpha^2\beta^2\left[ [ \beta(x),y   ],\beta(z)\right] 
		+\varepsilon(x,y)\alpha^2\beta^2\left[   \beta(y),[\alpha(x) ,z  ] 	\right]\\
		&= \alpha^2\beta^3\left[ [ x,y   ],z\right] 
		+\varepsilon(x,y)\alpha^2\beta^3\left[   y,[x ,z  ] 	\right]\\ 
		&= \alpha^2\beta^3\left( \left[ [ x,y   ],z\right] 
		+\varepsilon(x,y)\left[   y,[x ,z  ] 	\right]\right) \\ 
		&= \alpha^3\beta^3\big( \left[  x ,[y,  z]       \right]\big)\\
		&= 	  \left\lbrace \alpha\beta(x),\{y,z \}    \right\rbrace	       	          
		\end{align*}  	 
		Then, $ (L,[\cdot,\cdot], \varepsilon,\alpha,\beta) $ be a left
		BiHom-Leibniz  colour algebra.
		\item 
		If   $ (L,[\cdot,\cdot],\varepsilon)$ is a symmetric  Leibniz  colour algebra :
		\begin{equation*}	
		\left\lbrace \beta(x),\{\alpha(y), \alpha(z)\}    \right\rbrace
		=\alpha^2\beta^2 \left[  \beta(x),[\alpha(y), \alpha(z)]    \right]  
		=\alpha^4\beta^3\big( \left[  x,[y, z]    \right] \big) \\
		=\alpha\beta\big( \left[  x,[y, z]    \right] \big) \\	 
		\end{equation*} 	  	 
		\begin{align*}	
		-\varepsilon(x,y+z)\left\lbrace \{\beta(y), \beta(z)\} ,\alpha(x)   \right\rbrace
		&=-\varepsilon(x,y+z)
		\alpha^2\beta^2 \left[  [\beta(y), \beta(z)]  ,  \alpha(x)\right] \\ 
		&=-\varepsilon(x,y+z)\alpha^3\beta^4\big( \left[  x,[y, z]    \right] \big) \\
		&=-\varepsilon(x,y+z)\alpha\beta\big( \left[  [y, z] ,x   \right] \big) \\
		&=\alpha\beta\big( \left[  x,[y, z]    \right] \big) \\
		&=\left\lbrace \beta(x),\{\alpha(y), \alpha(z)\}    \right\rbrace.	 
		\end{align*} 
		Then,   $ (L,[\cdot,\cdot], \varepsilon,\alpha,\beta) $ is a symmetric
		BiHom-Leibniz  colour algebra.
		\item 
		If   $ (L,[\cdot,\cdot],\varepsilon)$ is a right  Leibniz  colour algebra : 
		Reasoning similarly as above proves that
		\begin{equation*}
		\left\lbrace \{x,y \} ,\alpha\beta(z),   \right\rbrace=\alpha\beta \big( \left[  x ,[y,  z]       \right]\big)        
		\end{equation*}
		and 
		\begin{equation*}
		\varepsilon(x,z)\left\lbrace \{x, \beta(z)\} ,\alpha(y)   \right\rbrace+\left\lbrace\alpha(x)  ,\{y,\alpha(z)\}\right\rbrace 
		=\alpha\beta \big( \left[ [x ,z],y      \right]+\left[x,[y,z]\right]  \big)        
		\end{equation*}
		Therefore, $ (L,[\cdot,\cdot], \varepsilon,\alpha,\beta) $ is a right
		BiHom-Leibniz  colour algebra.  
	\end{enumerate}   	 
\end{proof}
\begin{defn}
	Let $  (L,[\cdot,\cdot], \varepsilon,\alpha,\beta)$ be a   Leibniz  colour algebra.	
	The set $C_{\gamma}(L)$ consisting of linear map $d$ of degree $\gamma$ with the property  
	\begin{align*}
	d\circ \alpha &=\alpha \circ d,\quad   d\circ \beta= \beta\circ d \\
	d([x,y])&=[d(x),y]=\varepsilon(\gamma,x)[x,d(y)]    ,\quad\forall x,y\in \mathcal{H}(L),          
	\end{align*}  
	is called the centroids of $L$.               
\end{defn}
In the following proposition, we construct 
BiHom-
Leibniz colour  algebras starting from a 
BiHom-
Leibniz  colour  algebra
and an  two even elements in its centroid.
\begin{prop}
	Let $ (L,[\cdot,\cdot], \varepsilon,\alpha,\beta) $ be a 
	BiHom-
	Leibniz colour  algebra, and $\theta$, $\theta'$ be two even  elements in the centroid
	of $L$ satisfies $\theta\circ \theta' =  \theta' \circ\theta$,    $\theta^2= \theta$ and  $\theta'^2= \theta'$.
	Set for $ x,y\in \mathcal{H}(L)$,   $  [x,y]^{\theta'}=[\theta'(x),y] $   .  Then 	
	$(L,[\cdot,\cdot], \varepsilon,\theta\circ\alpha,\theta\circ\beta)$	
	and $	(L,[\cdot,\cdot]_1^{\theta'}, \varepsilon,\theta\circ\alpha,\theta\circ\beta) $
	are colour 
	BiHom-
	Leibniz algebras.	
\end{prop}
\section{Representations and cohomology of  (Bi)Hom-Leibniz algebras of type $B_1$}
First, we define a new type of BiHom-Lie algebras.  We call it BiHom-Lie algebras of type $B_1$. With the BiHom-Lie algebra of type $B_1$ we have the following hierarchy of algebras:
\begin{equation*}
\displaystyle   \{ \text{BiHom-Lie type $B_1$}  \}\supseteq_{\beta=id}   \{ \text{Hom-Lie}  \}\supseteq_{\alpha=id} \{ \text{Lie}  \}. 
\end{equation*}
For the rest of this article,we mean by  $ (L,[\cdot,\cdot],\alpha,\beta) $ a 
$4$-tuple consisting of   $ \K $-linear
space       $ L $, two linear maps $ \alpha,  $  $ \beta\, \colon L\to  L $      and  a bilinear map $[\cdot,\cdot] \colon L\times L\to  L  $
, satisfying the
following conditions
\begin{gather*}
\alpha\circ\beta=\beta\circ\alpha,\quad 
\alpha\left(  [x,y] \right) =\left[ \alpha(x),\alpha(y) \right] \text{\quad and  \quad  } 
\beta\left(  [x,y] \right) =\left[ \beta(x),\beta(y) \right]
\end{gather*}
for all $ x,\, y,\, z\in L $.

\begin{defn}\label{ bihom}
	The $4$-tuple $ (L,[\cdot,\cdot],\alpha,\beta) $ is a  BiHom-Lie  algebra of type $B_1$ if
		\begin{align*}
&	\left[ \beta(x),\beta^2(y)\right] =-\left[ \beta(y),\beta^2(x) \right], \text{ \quad   \quad }\\
&	\displaystyle \circlearrowleft_{x,y,z} \big[\alpha(x) ,[ \beta(y),\beta^2(z) ]  \big]=0 .  
	\end{align*}
	for all $ x,\, y,\, z\in L $.
\end{defn}
\begin{remq}
	Obviously, a BiHom-Lie algebra $ (L,[\cdot,\cdot],\alpha,\beta) $  of type $B_1$  for which $\beta=Id_{L} $ is just a
	Hom-Lie algebra $ (L,[\cdot,\cdot],\alpha) $.                             
\end{remq}
\begin{defn}
The  $4$-tuple $ (L,[\cdot,\cdot],\alpha,\beta) $ 
 is called a  left (resp. right) BiHom-Leibniz algebra of type $B_1$  if it satisfies the identity
	\[  \displaystyle  \big[\alpha(x) ,[ \beta(y),\beta(z) ]  \big]=\big[[x , \beta(y)],\alpha(z)   \big] +   
	\big[\alpha(y) ,[\beta(x),\beta(z)]   \big]
	\]
	respectively
	\[  \displaystyle \big[[x,y ],\alpha(z)   \big] =\big[[x,z] , \alpha(y)  \big] +   
	\big[\alpha(x),[y,\beta(z)]   \big].
	\]	  	   
\end{defn}

\begin{prop}
	If $ (L,[\cdot,\cdot],\alpha,\beta) $ is a left BiHom-Leibniz algebra of type $B_1$,
	then \[ 	\left[[y,\beta(x)],\alpha(z)\right]
	=-\left[[x,\beta(y)],\alpha(z)\right]   \]
\end{prop}
\begin{prop}
	If $ (L,[\cdot,\cdot],\alpha,\beta) $ is a right BiHom-Leibniz algebra of type $B_1$,
	then \[ 	\left[\alpha(x),[z,\beta(y)]\right]
	=-\left[\alpha(x),[y,\beta(z)]\right]   \]
\end{prop}
\begin{defn}
	$(L,[\cdot,\cdot],\alpha,\beta) $ is called 	
	a symmetric
	BiHom-Leibniz algebra of type $ B_1 $ if it is a left and a right BiHom-Leibniz algebra  of type $  B_1$.	
\end{defn}
\begin{prop}
	A $4$-tuple    $ (L,[\cdot,\cdot],\alpha,\beta) $ is a		symmetric BiHom-Leibniz algebra of type $ B_1 $	
 if and only if it satisfies 
	\begin{align*}
	& \displaystyle \big[[x,y ],\alpha(z)   \big] =\big[[x,z] , \alpha(y)  \big] +   
	\big[\alpha(x),[y,\beta(z)]   \big];\\
	&	\left[\alpha(y),[\beta(x),\beta(z)]\right]
	=-\left[ [x,z],\alpha\beta(y)\right] .    
	\end{align*}	    
\end{prop}
\subsection{  Representations of BiHom-Leibniz algebras  of type $ B_1 $ }
Lie algebra cohomology was introduced by Chevalley and Eilenberg \cite{Chevalley}. For Hom-Lie
algebra, the cohomology theory has been given by \cite{AbdenacerSil,Sheng}. A cohomology of BiHom-Lie 
algebras were introduced and investigated in \cite{Yongsheng}. 
We refer the reader to \cite{Loday2,Loday3,Pirashvili,Cheng1} for more information about  Leibniz
representations and Leibniz cohomologies.\\
In the following we
 define  a representations
of    (Bi)Hom–Leibniz algebras of  type $B_1$  and the corresponding coboundary operators.
We show that one can obtain the direct sum symmetric  (Bi)Hom–Leibniz algebras $(L\oplus V,[\cdot,\cdot]_f,\alpha+\alpha_V,\beta+\beta_V)$ of  type $B_1$.\\



\begin{defn}
	Let $V$ be a vector space,  $ \alpha_V,\, \beta_{V}\in End(V) $ and  $r,l\colon L \to End(V)$
	be two  linear maps satisfying
		\begin{align*}
	\alpha_V	\circ l(x)=l(\alpha(x))\circ \alpha_V&;\quad
	\alpha_	V\circ r(x)=r(\alpha(x))\circ\alpha_V ;\\
	\beta_V	\circ l(x)=l(\alpha(x))\circ \beta_V&;\quad
\beta_V	\circ r(x)=r(\alpha(x))\circ\beta_V ;
	\end{align*}
	\begin{enumerate}[(i)]	
	\item If $ (L,[\cdot,\cdot],\alpha,\beta) $ is a left BiHom-Leibniz algebra of type $B_1$, then we say that $(r,l)$ is a left representation of
$L$ in $V$ if for all $x,\, y\in L,\, v\in V$:		
\begin{align*}	
		l([x,\beta(y)])\circ\alpha_V &=l(\alpha(x))\circ l(\beta(y))\circ\beta_V-l(\alpha(y))\circ l(\beta(x))\circ\beta_V ;\\ 
		r\circ \beta([x,y])\circ\alpha _V&=l(\alpha(x))\circ r(\beta(y))\circ\beta_V-r(\alpha(y))\circ l(x)\circ\beta_V ;\\ 
			r\circ \beta([x,y])\circ \alpha_V&=r\left( \alpha(y)\right)\circ r(\beta(x))+l \left( \alpha(x)\right)\circ  r(\beta(y)) \circ\beta_V  .
		\end{align*}	
		\item If $L$ is a right BiHom-Leibniz algebra of type $B_1$, then we say that $(r,l)$ is a right representation of
		$L$ in $V$ if for all $x,\, y\in L,\, v\in V$:
		\begin{align*}
		l([x,y])\circ \alpha_V &=l(\alpha(x))\circ l(y)\circ \beta_V+            r(\alpha(y))\circ l(x);\\ 
			l([x,y])\circ \alpha_V&=    r(\alpha(y))\circ l(x)-                l(\alpha(x))\circ r(\beta(y));\\ 
		r([x,\beta(y)])\circ \alpha_V&=r \left( \alpha(y)\right)\circ  r(x)-  r\left( \alpha(x)\right)\circ r(y).
		\end{align*}	 	
	\end{enumerate}                                                           
\end{defn}
\begin{example}
	Let  $(L,[\cdot,\cdot],\alpha,\beta) $  be a left (resp. right ) BiHom-Leibniz algebra of type $B_1$. Then, we
	consider the maps $ad_{k}\colon L\to End(L) $ and $Ad_{k,l}\colon L\to End(L) $ 
	Defined by 		
	by  	
	$ad_{k,l}(a)(x)=[x,\alpha^k\beta^{l}(a)] $,  $Ad_{k}(a)(x)=[\alpha^k\beta^{l}(a),x]$, $\forall x\in L$. Therefore, $(ad_{k,l},Ad_{k,l})$ is a left representation (resp. a right
	representation ) of $L$ in $L$ called the left (resp. right) $\alpha^k\beta^{l}$-adjoint representation of $L$.
\end{example}
In the rest  of this article, if $r,l\colon L \to End(V)$, we denote  $r(x)(v)=[v,x]_V$ and $l(x)(v)=[x,v]_V$, $\forall x\in L,\, v\in V$ and  if $(r,l)$ is a representation  of $L$  we say $(V,[\cdot,\cdot]_{V},\alpha_{V},\beta_{V})$ is a representation  of $L$ (or a $L$-module).
\begin{prop}
		Let  $(L,[\cdot,\cdot],\alpha,\beta) $  be a left  BiHom-Leibniz algebra of type $B_1$ and $(V,[\cdot,\cdot]_{V},\alpha_{V},\beta_{V})$ be a left representation of $L$ . Then \[ \left[[v,\beta(x)]_{V},\alpha(y)\right]_{V}=- \big[[x ,\beta_V (v) ]_{V},\alpha(y)   \big]_{V}  \]  for all $x,\,y\in L,\, v\in V$.                                                       
\end{prop}
\begin{prop}
	Let  $(L,[\cdot,\cdot],\alpha,\beta) $  be a right  BiHom-Leibniz algebra of type $B_1$  and $(V,[\cdot,\cdot]_{V},\alpha_{V},\beta_{V})$ be a right representation of $L$. Then
\[  	\big[\alpha(x),[y,\beta_V(v)]_{V}   \big]_{V}=- 	\big[\alpha(x) ,[v,\beta(y)]  \big]_{V}                                \]	
	
for all $x,\,y\in L,\, v\in V$. 	                                                   
\end{prop}
\begin{defn}
	Let  $(L,[\cdot,\cdot],\alpha,\beta) $  be a left (resp. right ) BiHom-Leibniz algebra of type $B_1$.	
Let $V$ be a vector space,  $ \alpha_V,\, \beta_{V}\in End(V) $ and  $r,l\colon L\to End(V)$
be two  linear maps.	
Then, we say that $(r, l)$ is a
representation of $L$ in $V$ if $(r, l)$ is a left and a right representation of $L$ in
$V$ .	
\end{defn}

\begin{defn}
	Let $ (L,[\cdot,\cdot],\alpha,\beta) $ be a 	 BiHom-Leibniz algebra of type $ B_1 $  .	
	A representation $ (V,[\cdot,\cdot]_{V},\alpha_{V},\beta_{V}) $ is called:
	\begin{enumerate}[(i)]	
		\item trivial if $[x,\beta_{V}(v)]_{V}=[v,\beta(x)]_{V}=0,\quad \forall v\in V, x\in L. $
		\item adjoint if $(V,[\cdot,\cdot]_{V},\alpha_{V},\beta_{V})= (L,[\cdot,\cdot],\alpha,\beta) $	
	\end{enumerate}	
\end{defn}
\begin{prop}
	 Let $ (L,[\cdot,\cdot],\alpha,\beta) $ be a 	symmetric BiHom-Leibniz algebra of type $ B_1 $.  A $4$-tuples  $(V,[\cdot,\cdot]_{V},\alpha_{V},\beta_{V})$  is  a symmetric  representation of $V$ if and only if.
	 \begin{align*}
	 \alpha_V([x,v]_{V})&=[\alpha(x),\alpha_V(v)]_{V};\qquad 
	 \alpha_V([v,x]_{V})=[\alpha_V(v),\alpha(x)]_{V};\\
	 \beta_V([x,v]_{V})&=[\beta(x),\beta_V(v)]_{V};\qquad 
	 \beta_V([v,x]_{V})=[\beta_V(v),\beta(x)]_{V};\\
	 \displaystyle  	\big[  \left[ v,x\right]_{V} ,\alpha(y) ]\big]_{V}
	 &=\big[[v,y]_{V}, \alpha(x)   \big]_{V} +   
	 \big[\alpha_{V} (v),\left[x ,\beta(y)\right]_{V}   \big]_{V};\\
	 \big[ [x,y], \alpha_V(v)\big]_{V}
	 &=\big[[x,v]_{V},\alpha(y)   \big]_{V} +   
	 \big[\alpha(x),[y,\beta_V(v)]_{V}   \big]_{v};\\
	 \big[  [x,v]_{V},\alpha(y)\big]_{V}
	 &=\big[[x,y],  \alpha_V(v)  \big]_{V} +   
	 \big[\alpha(x) ,[v,\beta(y)]_{V}  \big]_{V};\\
	 \left[\alpha(x),[\beta(y),\beta_{V}(v)]\right]_{V}
	 &=-\left[ [y,\beta_{V}(v)],\alpha\beta(x)\right]_{V};\\
	 \left[\alpha(x),[\beta_{V}(v),\beta(y)]\right]_{V}
	 &=-\left[ [v,y]_{V},\alpha\beta(x)\right]_{V};\\
	 \left[\alpha(x),[\beta(y),\beta_{V}(v)]\right]_{V}
	 &=-\left[ [y,v]_{V},\alpha\beta(x)\right]_{V}.
	 \end{align*}                                  
\end{prop}
\begin{remq}
 If $(V,[\cdot,\cdot]_{V},\alpha_{V},\beta_{V})$  is  a  symmetric representation of a symmetric BiHom-Leibniz algebra
 $ (L,[\cdot,\cdot],\alpha,\beta) $ 	 of type $ B_1 $. Let
 us consider two linear maps $Ad_{n,m},ad_{n,m}\colon L\to End(V)$ defined by 
 \[  ad_{n,m}(x)(v) =[x,\alpha^{n}\beta^{m}(v)],\quad Ad_{n,m}(x)(v) =[\alpha^{n}\beta^{m}(v),x],\quad    \forall  x\in L,\, v\in V.                                                \]
 Then $ (Ad_{n,m},ad_{n,m})$ is symmetric representation of $L$ in $ V $.                            
\end{remq}

\subsection{  Cohomology   of 	symmetric BiHom-Leibniz algebras  of type $ B_1 $ }
Let $ (L,[\cdot,\cdot],\alpha,\beta) $ be a 	symmetric BiHom-Leibniz algebra of type $ B_1 $  and $(V,[\cdot,\cdot]_{V},\alpha_{V},\beta_{V})$  be a symmetric representation of $V$. Denote \[    C^k(L,V) =Hom(L^k,V),\qquad k   \geq    0,  \]
and \[   C^k_{(\alpha,\alpha_V)}(L,V)=\{ f\in C^k(L,V)\mid\, f\circ \alpha=\alpha_V \circ f      \}    .  \]
Let $\delta^k:C^k_{(\alpha,\alpha_V)}(L,V)\to    C^{k+1}_{(\alpha,\alpha_V)}(LV,)       $ be a $ k $-Homomorphism  defined by 
\begin{gather*}
\delta^{k}_{n,m}(f)(x_{0},\dots,x_{k}) \nonumber\\
=\sum_{0< t\leq k}(-1)^{t+1}f
\Big([x_{0},x_t],\cdots,\widehat{x_{t}},\cdots,\alpha(x_{k})\Big)\\
+ 
\sum_{0< s < t\leq k}(-1)^{t+1}
(f)
\Big(\alpha(x_{0}),\cdots,\alpha(x_{s-1}),[x_{s},\beta(x_{t})],\alpha(x_{s+1}),\cdots,\widehat{x_{t}},\cdots,\alpha(x_{k})\Big)  \label{ cobbTypeA}
\end{gather*}
\begin{gather*}
-\Big[\alpha^{n+k-1}\beta^{m+k-1}(x_{0}), f\left( x_{1},\beta(x_2),\cdots,\beta(x_k)\right)  \Big]_{V}\\
+\sum_{s=1}^{k}(-1)^{s}\Big[f\left( x_{0},\cdots,\widehat{x_{s}},\dots,x_{k}\right) ,\alpha^{n+k-1}\beta^{m+k-1}(x_{s}) \Big]_{V}\label{def cobbbb}
\end{gather*}
where $ x_0,\cdots,x_k\in L $ and $ \hat{x}_{i} $ means that $ x_{i} $
is omitted.
The pair $(\oplus_{k>0}C^{k}_{\alpha ,\alpha_V)}(L,\ V),\{\delta^{k}\}_{k>0})$ defines a chomology complex, that is
$\delta^{k} \circ \delta^{k-1}=0.$

\begin{itemize}
	\item The $k$-cocycles space is defined as $Z^{k}(L, V)=\ker \ \delta^{k}.$
	\item  The $k$-coboundary space is defined as    $B^{k}(L, V)=Im \ \delta^{k-1}.$
	\item The $k^{th}$ cohomology  space is the quotient  $H^{k}(L, V)=Z^{k}(L, V)/ B^{k}(L, V). $ \\
\end{itemize}

	If $(ad_{n,m},Ad_{n,m}) $ is a symmetric adjoint representation of $L$. Then any $1$-cocycle    $f\in Z^{1}(L, L)$    is called a $\alpha^{n}\beta^{m}$-derivation of $L$.  

\subsection{ Extensions of                 BiHom-Leibniz algebras  of type $ B_1 $            }

In this section we extend extensions theory of Leibniz algebras introduced in \cite{Loday3}  to  BiHom-Leibniz algebras  of type $ B_1 $        case.

An extension of a   BiHom-Leibniz algebras   $(L,[\cdot,\cdot],\alpha,\beta)$  of type $ B_1 $ by $L$-module  $(V,[\cdot,\cdot]_{V},\alpha_V,\beta_V)$ is an exact sequence

$$0\longrightarrow (V,\alpha_V,\beta_{V})\stackrel{i}{\longrightarrow} (\widetilde{L},\widetilde{\alpha},\widetilde{\beta})\stackrel{\pi}{\longrightarrow }(L,\alpha,\beta) \longrightarrow 0 $$
satisfying $\widetilde{\alpha}\circ  i=i\circ \alpha_V  $,  $\widetilde{\beta}\circ  i=i\circ \beta_V  $, $\alpha \circ \ \pi = \pi\  o \ \widetilde{\alpha}$  and $\beta \circ  \pi = \pi\circ \ \widetilde{\beta}.$\\
We say that the extension is central if $[\widetilde{\mathcal{G}},i(V)]_{\widetilde{\mathcal{G}}}=0.$\\
Two extensions $$0 \longrightarrow  (V_{k},\alpha_{V_k},\beta_{V_k})\stackrel{i_{k}}{\longrightarrow}      (L_{k},\alpha_{k},\beta_{k})\stackrel{\pi_{k}}{\longrightarrow } (L,\alpha,\beta) \longrightarrow 0 \ \ \ (k=1,2)$$ are equivalent if there is an isomorphism $\varphi\colon(L_{1},\alpha_{1},_{1})\rightarrow (L_{2},\alpha_{2},\beta_{2})$ such that $\varphi \circ i_{1}= i_{2}$ and $\pi_{2}\circ \varphi=\pi_{1}.$   One denote by $Ext(L,V)$ the set of isomorphism classes of extensions of $L$ by $V$.  In the sequel, we  assume that $ i(V) $ is of finite codimension in $ \widetilde{L} $ .\\

Let $f \in C^{2}_{(\alpha ,\alpha_V)}(L,\ V).$ Assume that $L\cap V=\{0\}$ and 
 we consider the direct sum $\widetilde{L}=L\oplus V $ with the following bracket $$ \big[(x,u), (y,v)\big]_{\widetilde{L}}=\big([x,y],[x,v]_{V}+[u,y]_{V}+f(x,y)\big);\ \ \ \forall x,y \in L, \     u, v\in V.$$
Define the linear maps $ \widetilde{\alpha}, \widetilde{\beta}\colon \widetilde{L} \to  \widetilde{L}$ by $\widetilde{\alpha}(x,v)=(\alpha(x),\alpha_V(v)) $  by $\widetilde{\beta}(x,v)=(\beta(x),\beta_V(v)) .$
\begin{lem}
 With the above notations,	the $4$-tuple  $(\widetilde{L},[\cdot,\cdot]_{\widetilde{L}},\widetilde{\alpha},\widetilde{\beta})$   is a symmetric BiHom-Leibniz of type $B_1$
	if and only if $f$ is a $2$-cocycle (i.e. $\delta^{2}(f)=0$).\\
\end{lem}
\begin{thm}
	  For any   symmetric BiHom-Leibniz of $L$ type $B_1$  and any    symmetric representation $V$ of $L$, there is a natural bijection 
	  \[     Ext(L,V)\cong  H^{2}(L, V) .      \]                  
\end{thm}
\section*{acknowledgements}
The author would like to thank professor Abdenacer Makhlouf for carefully reading and checking the paper and
his useful suggestions and comments.

\end{document}